\documentclass[10pt,english]{amsart}
\usepackage{amsfonts, amssymb, amsmath, amsthm, eucal, latexsym}

\usepackage[english]{babel}

\usepackage{epsfig}
\usepackage{amsmath}
\usepackage{color}
\usepackage{graphicx}
\def\div{\hbox{div}}

\newcommand{\ba}{\begin{eqnarray}}
\newcommand{\ea}{\end{eqnarray}}

\newtheorem{thm}{Theorem}[section]

\newtheorem{theorem}[thm]{Theorem}
\newtheorem{definition}[thm]{Definition}
\newtheorem{lemma}[thm]{Lemma}
\newtheorem{proposition}[thm]{Proposition}
\newtheorem{corollary}[thm]{Corollary}
\newtheorem{rem}[thm]{Remark}

\numberwithin{equation}{section}

\makeatletter

\newcommand{\Rmnum}[1]{\expandafter\@slowromancap\romannumeral #1@}
\makeatother

\title{\bf ASYMPTOTICS OF THE CRITICAL NON-LINEAR WAVE EQUATION FOR A CLASS OF NON STAR-SHAPED OBSTACLES}
\author{ Farah ABOU SHAKRA }
\thanks{The author is a Lebanese CNRS scholar.}
\address{Universit\'{e} Paris 13, Sorbonne Paris Cit\'{e}, LAGA, CNRS (UMR 7539), F-93430, Villetaneuse, France.
}

\email{shakra@math.univ-paris13.fr}

\begin{document}

\begin{abstract}
Scattering for the energy critical non-linear wave equation for domains exterior to non trapping obstacles in 3+1 dimension is known for the
star-shaped case. In this paper, we extend the scattering for a class of non star-shaped obstacles called illuminated from exterior. The main tool we use is the method of multipliers with
weights that generalize the Morawetz multiplier to suit the geometry of the obstacle.
\end{abstract}

\maketitle

\section{Introduction}
 In this paper we are working on the energy critical nonlinear wave equation in 3+1 dimension in a domain $\Omega=\mathbb{R}^3\setminus V$ 
where $V$ is a non-trapping obstacle with smooth boundary
\begin{align}\label{waveequation}
&\Box u=(\partial^{2}_{t}-\Delta)u=-u^{5}\;\;in\;\; \mathbb{R}\times\Omega\notag\\
&u|_{\mathbb{R}\times\partial\Omega}=0\\
&(\nabla u(t,\cdot),\partial_{t}u(t,\cdot))\in L^{2}(\Omega)\;\;\;t\in\mathbb{R}\notag
\end{align}
which enjoys the conservation of energy
 $$E=E(t)=\int_{\Omega}\frac{|\partial_{t}u|^{2}}{2}+\frac{|\nabla u|^{2}}{2}+\frac{|u|^{6}}{6}dx.$$
In the boundaryless case ($\Omega=\mathbb{R}^3$), the first results for the global existence were obtained by Grillakis (\cite{G90}, \cite{G92}). 
He showed that there are global smooth solutions of the critical wave equation, if the data is smooth. Shatah and Struwe (\cite{SS93}, \cite{SS94}) 
extended this theorem by showing that there are global solutions for the data lying in the energy space $H^1\times L^2$. 
They also obtained results for critical wave equation in higher dimensions.

For the case of obstacles, the first results were due to Smith and Sogge \cite{SS95}. They showed that Grillakis theorem extends 
to the case where $\Omega$ is the complement of a smooth, compact, strictly convex obstacle with Dirichlet boundary conditions. 
This result was later extended to the case of arbitrary domains in $\mathbb{R}^3$ and data in the energy space by Burq, Lebeau and Planchon
\cite{BLP08}. The case of the nonlinear critical Neumann wave equation in 3-dimensions was subsequently handled by Burq and Planchon
\cite{BP09}.

More specifically, in this paper we are interested in asymptotics, i.e. how solutions to the nonlinear equation scatter to a solution to 
the homogeneous linear equation.
In the boundaryless case ($\Omega=\mathbb{R}^{3}$), first results were obtained by Bahouri-G\'{e}rard in \cite{BG99}; in their 
paper they used the following decay estimate proved by Bahouri-Shatah \cite{BS98}
\begin{equation*}
\lim_{|t|\rightarrow+\infty}\frac{1}{6}\int_{\Omega}|u(t,x)|^{6}dx=0.
\end{equation*}
to get
\begin{equation*}
\left\|u\right\|_{L^{5}(\mathbb{R};L^{10}(\Omega))}+\left\|u\right\|_{L^{4}(\mathbb{R};L^{12}(\Omega))}<\infty,
\end{equation*}
and thus scattering. Moreover, in \cite{BG99} Bahouri-G\'{e}rard used profile decomposition to show that $\left\|u\right\|_{L^{5}(\mathbb{R};L^{10}(\Omega))}$ 
is also controlled by a universal function of the energy $f(E)$.\newline
Then the scattering result was extended to the case of star-shaped obstacles ($x\cdot n\geq 0$ for $x\in\partial V$ with $n$ the outward 
pointing unit normal vector to $\partial V$) by Blair, Smith, and Sogge in \cite{BSS09}. 
They used the same $L^{6}$ decay estimate proved by Bahouri-Shatah in \cite{BS98} after they extended it to their case of obstacles 
making slight modifications on the proof to handle the boundary term.

In the papers by Bahouri-Shatah \cite{BS98} and Blair, Smith, and Sogge \cite{BSS09}, the $L^{6}$ decay estimate
which is the main key to prove scattering was proved using the method
of multipliers. The method of multipliers is also called Friedrichs' ABC method as it dates back to Kurt O. Friedrichs in the 1950's. The idea of this method is to multiply the equation with a factor $Nu$, with $N$ is a linear first-order 
differential operator, defined as
$$Nu=Au +B\cdot\nabla u+C\partial_{t}u$$ 
and then to express the product as a divergence or energy identity of the form $$\div_{t,x}(\cdots)+Remaining\;terms=0$$
and finally to integrate this divergence identity over a domain in $\mathbb{R}^{n+1}$ and subsequently derive the required estimates.
The only case where the differential multiplier is adapted to both the
wave equation in terms of commutation (avoiding remaining terms) and the geometry of the obstacle in terms of
the sign of the boundary term, is the star-shaped case.
The method of multipliers was used in the 1960's and 1970's to prove uniform decay results for the homogeneous 
linear wave equation ($\Box u=0$) outside obstacles. Cathleen S. Morawetz was the first to succeed in proving uniform local energy decay for star-shaped obstacles
with Dirichlet boundary condition using this method (\cite{M61} and \cite {M62}). Since then, the results of Morawetz have 
been considerably improved. Better decay rates have been achieved (as in odd dimensions $n\geq 3$, Huygen's principle has been shown to imply an exponential rate of decay whenever there is some sort of decay \cite{LMP63}, \cite{M66}). 
Moreover, the class of obstacles under consideration has been enlarged; decay results have been derived for a special case of non trapping obstacles referred to as \textquotedblleft almost star-shaped regions\textquotedblright (Ivrii \cite{Ivrii}) and for non trapping obstacles with simple and direct 
geometrical generalizations to the star-shaped such as the \textquotedblleft illuminated from interior\textquotedblright (Bloom and Kazarinoff \cite{BK74}) 
and the \textquotedblleft illuminated from exterior\textquotedblright (Bloom and Kazarinoff \cite{BK76}, Liu \cite{L87}). For these cases, decay results have been proved using the method of 
multipliers after generalizing the multipliers to suit the geometry of the obstacle, and although these generalized multipliers lead to volume integrals that were 
avoided before, it turned out that these integrals were actually useful in 
the estimates. Other wider generalizations that include all the above geometries later followed, Strauss \cite{S75} proved uniform local energy decay for the homogeneous linear wave equation in exterior domains in $\mathbb{R}^{n}\;n\geq 3$, provided 
a strictly expansive vector field (now called the Straussian vector field) exists, that leaves $\overline{\Omega}$ strictly invariant, then these Straussian vector fields were 
generalized by Morawetz, Ralston, and Strauss \cite{MRS77} by introducing escape functions, using them to construct a pseudo-differential operator 
$P(x,D)$, and finally setting $Pu$ as a multiplier.

Though the microlocal methods of Morawetz, Ralston, and Strauss provided general results for the linear case, they cannot be 
easily extended to the nonlinear wave equation for which the star-shapedness has been so far a restriction to obtain decay results. Thus, we consider in this 
paper obstacles with explicit geometry that is a direct generalization of the star-shapedness, and we prove the $L^{6}$ decay estimate and thus the scattering for such obstacles using Friedrichs' ABC method after generalizing the multipliers to suit our case. The price to pay 
for such a generalization is that the our multiplier, which is adapted to the geometry of the obstacle, no longer has
exact commutation properties with the wave operator. Therefore, unlike the star-shaped case we do get volume integrals that we deal with using a Gronwall-like argument.

Our main results are:
\begin{theorem}\label{L6decay}
Suppose $u$ solves the nonlinear wave equation \eqref{waveequation} and that $V$ is a non-trapping obstacle with regular boundary that can be illuminated from its exterior by a 
strictly convex body $C$ satisfying the geometric condition: 
\begin{equation}\label{8}
\min_{\partial V}(s_{0}+\rho_{1}-2(\rho_{2M}-\rho_{1}))>0
\end{equation}
where 
\begin{itemize}
 \item $s_{0}$ is the algebraic distance from $\partial C$ to $x\in\partial V$ along the exterior normal to $\partial C$.
 \item $\rho_{2M}=\displaystyle\max_{(\sigma_{1},\sigma_{1})}\rho_{2}$ where $\rho_{i}$ ($i=1,2$) are the radii of curvature of $\partial C$ ($\rho_{2}\geq\rho_{1}$).
\end{itemize}
then  $$\lim_{t\longrightarrow\infty}\int_{\Omega}|u(t,x)|^{6}dx=0$$
\end{theorem}
\begin{rem}
We can construct obstacles that are illuminated from the exterior or from the interior that satisfy the condition \eqref{8}. In particular, it would be easy to see this for illuminated from interior obstacles, where the illuminating body is inside the obstacle and thus $s_{0}>0$, by considering a dog bone like obstacle (Figure \ref{fig:dogbone}) that is a small perturbation of the star-shaped. 
\end{rem}
\begin{rem}
 Remark that if the data has compact support, the computation that proves the above result provides an explicit decay rate for the local energy. In particular, it recovers Bloom-Kazarinoff for the linear
equation (\cite{BK74}) without using the fact that $\Box \partial_{t}u=0$.
\end{rem}

As a result of this decay estimate, we get scattering
\begin{corollary}\label{scattering}
Suppose $u$ solves the nonlinear wave equation \eqref{waveequation} and that $V$ is a non-trapping obstacle with regular boundary that can be illuminated from its exterior 
by a strictly convex body satisfying the geometric condition:
\begin{equation*}
\min_{\partial V}(s_{0}+\rho_{1}-2(\rho_{2M}-\rho_{1}))>0
\end{equation*}
then there exists unique solutions $v_{\pm}$ to the homogeneous linear problem 
 \begin{equation}\label{linearequation}
\left\{
\begin{aligned}
&\Box v=0\;\;in\;\; \mathbb{R}\times\Omega\\
&v|_{\mathbb{R}\times\partial\Omega=0}
\end{aligned}
\right.
\end{equation}
such that
\begin{equation}\label{energyzero}
\lim_{t\rightarrow\infty\pm}E_{0}(u-v_{\pm};t)=0.
\end{equation}
Moreover, $u$ satisfies:
\begin{equation}\label{spacetimeestimate}
\left\|u\right\|_{L^{5}(\mathbb{R};L^{10}(\Omega))}+\left\|u\right\|_{L^{4}(\mathbb{R};L^{12}(\Omega))}<\infty.
\end{equation}
\end{corollary}
{\bf Acknowledgements}. I would like to thank Fabrice Planchon for suggesting the problem and commenting on the manuscript.
\section{The geometry of the obstacle}
We consider in this paper illuminated from exterior obstacles, defined as such (Liu \cite{L87}):
\begin{definition}\label{exterior}
 We say that the boundary of an exterior domain $\Omega=\mathbb{R}^{3}\setminus V$ (or the obstacle $V$) can be illuminated from the exterior if and only if there exists a convex body $C$
containing $\partial V$ with smooth boundary $\partial C$ such that $\partial V$ is filled by a family of non-intersecting rays normal
to $\partial C$. Each ray is completely contained in $\Omega$ in the following sense: for each $x_{0}\in\partial V$ there exists a unique $x_{1}\in\partial C$
and a number $s_{0}(x_{1})\leq0$ such that
$$x_{0}=s_{0}(x_{1})\nu(x_{1})+x_{1},$$
where $\nu$ is the outward unit normal to $\partial C$ at $x_{1}$, and 
$$x=t\nu(x_{1})+x_{1}\in\Omega,\;\;\;\;\;\;t\in[s_{0},\infty).$$
\end{definition}
This definition actually generalizes the following definition of illuminated from interior obstacles introduced by Bloom and Kazarinoff \cite{BK74}:
\begin{definition}\label{illuminated}
We say that a body $V$ can be illuminated from the interior if and only if there exists a smooth convex body $C$ inside $V$ such 
that $extC$ is filled by a family of non-intersecting rays normal to $\partial C$ and such that each ray intersects $\partial V$ exactly 
once. 
\end{definition}
In fact, every body that can be illuminated from the interior can also be illuminated from its exterior by enlarging the original convex body 
(Liu \cite{L87}); thus our results proven for illuminated from exterior obstacles also hold for illuminated from interior. Furthermore, as we 
mentioned in the introduction, these geometries are direct generalizations of the star-shaped. 
More precisely, the illuminated from interior is a generalization of the strict star-shapedness ($x\cdot n>0$). The condition $x\cdot n>0$
implies that each ray
beginning at the origin intersects $\partial V$ exactly once, which means that the interior of a strictly star-shaped obstacle
can be illuminated by a source of light situated at the origin. A small ball centered at the origin is contained in the interior of $V$ and
the above light rays
are perpendicular to the surface of this ball, hence our strictly star-shaped is illuminated from its interior by this ball.\newline
An example of a non star-shaped body that can be illuminated from interior is a \textquotedblleft dog bone\textquotedblright (Figure \ref{fig:dogbone}), and a \textquotedblleft snake-shaped\textquotedblright body (Figure \ref{fig:snake}) is an example of an obstacle
that cannot be illuminated from its interior but can be illuminated from the exterior.
\begin{figure}[htbp]
\centering
 \includegraphics[width=8cm]{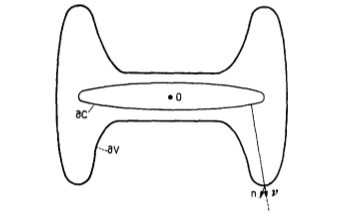}
 \caption{dog bone}
\label{fig:dogbone}
\end{figure} 

\begin{figure}[htbp]
\centering
\includegraphics[width=8cm]{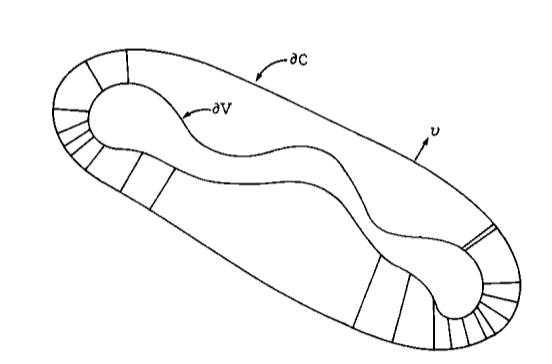}
\caption{snake}
\label{fig:snake}
\end{figure}
\subsection{The illuminating coordinate system}
In this section we will introduce the coordinate system that we are going to use which is the one in the paper by Liu \cite{L87} in which he 
was dealing with similar obstacle for the linear equation. 
We again denote by $\partial C$ the smooth and convex surface of the illuminating body.
Let $x=(x_1,x_2,x_3)$ be Cartesian coordinates in $\mathbb{R}^{3}$ with the origin inside $C$ and $V$. If $X^{0}$ is on $\partial C$,
then in a neighborhood of $X^{0}$ we choose the parametric curves to be the two principal curves on $\partial C$. If the neighborhood of $X^{0}$
is an all-umbilic surface, then we still can choose the parametric curves to be orthogonal to each other. Furthermore, we let the parameters be 
the arc-length parameters. Thus, if $X^{0}\in\partial C$, then $X^{0}$ is given in local coordinates by 
$$X^{0}=X^{0}(\sigma_{1},\sigma_{2})=\left(X^{01}(\sigma_{1},\sigma_{2}),X^{02}(\sigma_{1},\sigma_{2}),X^{03}(\sigma_{1},\sigma_{2})\right),$$
where $\sigma_{1}=const.$ and $\sigma_{2}=const.$ are the parameterizations of the arc-length of the principal curves near $X^{0}$. 
A finite number of $(\sigma_{1},\sigma_{2})$ coordinate patches cover $\partial C$.
Next, corresponding to each point $X^{0}(\sigma_{1},\sigma_{2})$ on $\partial C$, we make the choice $x=X(s,\sigma_{1},\sigma_{2})$, where
\begin{equation}\label{2}
x=X(s,\sigma_{1},\sigma_{2})=s\nu(\sigma_{1},\sigma_{2})+X^{0}(\sigma_{2},\sigma_{2})=sX_{s}+X^{0}(\sigma_{1},\sigma_{2}),
\end{equation}
where $$\nu(\sigma_{1},\sigma_{2})=\left(\frac{X^{0}_{\sigma_{1}}}{|X^{0}_{\sigma_{1}}|}\times\frac{X^{0}_{\sigma_{2}}}{|X^{0}_{\sigma_{2}}|}\right)$$
is the unit exterior normal to $\partial C$, with $X_{\sigma_{i}}^{0}\equiv\partial X^{0}/\partial\sigma_{i},\;i=1,2$. By Definition \ref{exterior}, for each $X^{1}\in\partial V$, 
there is a unique triple $(s_{0},\sigma_{1},\sigma_{2})$ with $s_{0}\leq 0$ such that 
$$X^{1}=s_{0}\nu(\sigma_{1},\sigma_{2})+X^{0}(\sigma_{1},\sigma_{2}),$$ \newline
We denote by $\kappa_{1}$ and $\kappa_{2}$ the principal curvatures at $X^{0}(\sigma_{1},\sigma_{2})$ and $\rho_{i}=\frac{1}{\kappa_{i}}\;(i=1,2)$ the radii of curvature of $\partial C$.
We assume $0<\kappa_{2}\leq \kappa_{1}$ ($\rho_{2}\geq\rho_{1}>0$).
Furthermore, we always assume that
\begin{equation}\label{s+rho}
\min_{\partial V}(s_{0}+\rho_{1})>0
\end{equation}
This condition implies that for every $x\in\Omega$, we have $s+\rho_{i}>0$ ($i=1,2$) since $x=s\nu(\sigma_{1},\sigma_{2})+X^{0}(\sigma_{2},\sigma_{2})$ with $s_{0}\leq s<\infty$, where $s_{0}$ corresponds to the point on $\partial V$ associated with $X^{0}$.

\begin{rem}
 Generically, Definition \ref{exterior} implies the condition \eqref{s+rho} (\cite{BK76} page 26 Lemma 2.1, \cite{L87} page 316 Remark after Lemma 1). Moreover, note that for a star-shaped obstacle, which is in fact an obstacle that is illuminated from the exterior by some ball $B(0,R_{0})$, $s+\rho_{i}$ is nothing but $r$. This explains the significance of this value and 
makes the need of such an assumption in a computation that is a generalization of the star-shaped totally logical.
\end{rem}

Now, we state the following geometrical lemma that we will use later:
\begin{lemma}\label{a0}
There exist a constant $a_{0}>0$ such that $s+\rho_{2M}\geq a_{0}r$.
\end{lemma}
\begin{proof}
The existence of $a_{0}$ is due to the boundedness of the the obstacle $V$ and the illuminating body $C$. In fact, if $s\geq0$ ($x\in\overline{extC}$) then $r-r_{0}(\sigma_{1},\sigma_{2})\leq s<r$ where $r_{0}(\sigma_{1},\sigma_{2})=|X^{0}(\sigma_{1},\sigma_{2})|$ and if $s<0$ ($x\in\overset{\circ}{C}\cap\overline{\Omega}$) then $s+\rho_{2M}\geq c$ for some positive constant $c$.
\end{proof}
We also recall the following lemmas about the coordinate 
system, these lemmas were originally stated and proved by Bloom and Kazarinoff \cite{BK74} for illuminated from interior obstacles, and then 
they were extended by Liu \cite{L87} for illuminated from exterior obstacles.
\begin{lemma}\label{lemma1}
The level surfaces $s=const.$, $\sigma_{i}=const.$ $(i=1,2)$ define a set of local coordinate systems in $\overline{\Omega}$ with 
each ray $\left\{x:\sigma_{1}=const.,\sigma_{2}=const.,\;and\; s_{0}\leq s<\infty\right\}$ normally incident on $\partial C$ and 
$$\left|\frac{D(x_{1},x_{2},x_{3})}{D(s,\sigma_{1},\sigma_{2})}\right|=\Lambda(\kappa_{1}s+1)(\kappa_{2}s+1)>0$$ where $\Lambda=|X^{0}_{\sigma_{1}}||X^{0}_{\sigma_{2}}|$.
\end{lemma}
\begin{lemma}\label{lemma2}
$\nu\cdot n\geq 0$ on $\partial V$.
\end{lemma}
\begin{rem}\label{expressions}
We recall the following calculus formulas that we will use in our computation $(i=1,2)$: 
$$\partial_{\sigma_{i}}\nu\equiv\frac{\partial\nu}{\partial\sigma_{i}}=\kappa_{i}X^{0}_{\sigma_{i}}$$ 
and thus
$$X_{\sigma_{i}}=(\kappa_{i}s+1)X^{0}_{\sigma_{i}}.$$
Moreover, remark that for any scalar function $f=f(s,\sigma_{1},\sigma_{2})$ and every vector field written in the new coordinate system 
$$F=F^{0}\nu+F^{1}\frac{X_{\sigma_{1}}^{0}}{\left|X_{\sigma_{1}}^{0}\right|}+F^{2}\frac{X_{\sigma_{2}}^{0}}{\left|X_{\sigma_{2}}^{0}\right|},$$ we can express the gradient and the divergence as follows:
$$\nabla f=\partial_{s}f\nu+\frac{1}{\left|X_{\sigma_{1}}\right|}\partial_{\sigma_{1}}f\frac{X^{0}_{\sigma_{1}}}{\left|X^{0}_{\sigma_{1}}\right|}+\frac{1}{\left|X_{\sigma_{2}}\right|}\partial_{\sigma_{2}}f\frac{X^{0}_{\sigma_{2}}}{\left|X^{0}_{\sigma_{2}}\right|}$$
and
$$\div F=\frac{1}{\left|X_{\sigma_{1}}\right|\left|X_{\sigma_{2}}\right|}\left[\partial_{s}\left(\left|X_{\sigma_{1}}\right|\left|X_{\sigma_{2}}\right|F^{0}\right)+\partial_{\sigma_{1}}\left(\left|X_{\sigma_{2}}\right|F^{1}\right)+\partial_{\sigma_{2}}\left(\left|X_{\sigma_{1}}\right|F^{2}\right)\right]$$
In particular, we have $$\nabla s=\nu$$ $$\nabla f\cdot\nu=\partial_{s}f$$
$$\left|\nabla f\right|^{2}=\left(\partial_{s}f\right)^{2}+\frac{1}{(\kappa_{1}s+1)^{2}|X^{0}_{\sigma_{1}}|^{2}}(\partial_{\sigma_{1}}f)^{2}+\frac{1}{(\kappa_{2}s+1)^{2}|X^{0}_{\sigma_{2}}|^{2}}(\partial_{\sigma_{2}}f)^{2}$$
Denote by
$$|\nabla^{*}_{i} f|^{2}=\frac{1}{(\kappa_{i}s+1)^{2}|X^{0}_{\sigma_{i}}|^{2}}(\partial_{\sigma_{i}}f)^{2},\;i=1,2$$
and
$$|\nabla^{*} f|^{2}=|\nabla^{*}_{1} f|^{2}+|\nabla^{*}_{2} f|^{2}$$
thus $$|\nabla f|^{2}=|\partial_{s}f|^{2}+|\nabla^{*} f|^{2}.$$
\end{rem}
\section{Proof of the $L^{6}$ decay estimate (Theorem \ref{L6decay})}
We must show that for any $\epsilon_{0}>0$, there exists $T_{0}$ such that whenever $t\geq T_{0}$, 
$$\frac{1}{6}\int_{\Omega}|u(t,x)|^{6}dx\leq\epsilon_{0}.$$
First, we begin by multiplying the wave equation $\Box u+u^{5}=0$ by $\partial_{t}u$, we get the following divergence or energy identity
$$\partial_{t}(e(u))-\div(\nabla u\partial_{t}u)=0$$
where $$e(u)=\frac{1}{2}(|\partial_{t}u|^{2}+|\nabla u|^{2})+\frac{1}{6}|u|^{6}$$ denotes the energy density. Integrating over the region $\{(x,t);s+\rho_{2M}>t+M,0\leq t\leq T\}$, where $\rho_{2M}=\displaystyle\max_{(\sigma_{1},\sigma_{2})}\rho_{2}$ 
and $M$ is a positive constant chosen such that $M\geq\rho_{2M}$ and the illuminating body $C\subset\{x;s+\rho_{2M}\leq M\}$, and using the divergence theorem and
the Dirichlet boundary condition, we get
\begin{equation}\label{exteriorcone}
\int_{s+\rho_{2M}>T+M}e(u)(T,x)dx+\frac{1}{\sqrt{2}}flux(0,T)\leq\int_{s+\rho_{2M}>M}e(u)(0,x)dx
\end{equation}
where the flux on the mantle is defined by:
$$flux(a,b)=\int_{M_{a}^{b}}\left(\frac{1}{2}\left|\nu\partial_{t}u+\nabla u\right|^{2}+\frac{u^{6}}{6}\right)d\sigma$$
with $$M_{a}^{b}=\left\{(x,t); s+\rho_{2M}=t+M, a\leq t\leq b\right\}$$
Since the solution has finite energy, we may select $M$ large so that the right hand side of \eqref{exteriorcone} is less than 
$\frac{\epsilon_{0}}{2}$. Hence, it will suffice to show the existence of $T_{0}$ such that whenever $T>T_{0}$ we have 
$$\frac{1}{6}\int_{s+\rho_{2M}\leq T+M}|u(T,x)|^{6}dx\leq\frac{\epsilon_{0}}{2}$$
This is a consequence of the following proposition:
\begin{proposition}\label{differentialinequality}
Suppose $u$ solves the nonlinear wave equation \eqref{waveequation} and that $V$ is a non-trapping obstacle with regular boundary that can be illuminated from its exterior 
by a strictly convex body satisfying the geometric condition:
\begin{equation*}
\min_{\partial V}(s_{0}+\rho_{1}-2(\rho_{2M}-\rho_{1}))>0
\end{equation*}
then
\begin{align*}
&\sqrt{\eta_{0}}\int_{s+\rho_{2M}\leq\epsilon T+M}\left(\left|\nabla^{*}u\right|^{2}+\left|\frac{\partial_{s}((s+\rho_{2M})u)}{s+\rho_{2M}}\right|^{2}\right)dx+\int_{s+\rho_{2M}\leq T+M}\frac{u^{6}(T,x)}{3}dx\\
&\;\leq 2c_{1}\beta E+\frac{1}{T}\left(C_{0}E+C_{2}E\ln(1+T)+2(c_{2}+c_{3}T)flux(0,T)\right)\\
&\;\;\;+\frac{\eta_{0}}{T}\int_{0}^{T}\int_{s+\rho_{2M}\leq\epsilon t+M}\left(|\nabla^{*}u|^{2}+\left|\frac{\partial_{s}((s+\rho_{2M})u)}{s+\rho_{2M}}\right|^{2}\right)dxdt\notag
\end{align*}
for some arbitrary $0<\beta<1$ and where $0<\eta_{0}, \epsilon<1$ and all the other constants depend on the geometry of the obstacle.
\end{proposition}
As a result of Proposition \ref{differentialinequality}, we will show that given $\epsilon_{0}>0$, $\exists T_{0}$ such that $\forall T\geq T_{0}$ we have
\begin{equation}\label{gronwall}
\frac{1}{T}\int_{0}^{T}\int_{s+\rho_{2M}\leq\epsilon T+M}\left(\left|\nabla^{*}u\right|^{2}+\left|\frac{\partial_{s}((s+\rho_{2M})u)}{s+\rho_{2M}}\right|^{2}\right)dxdt<\frac{\epsilon_{0}}{2}.
\end{equation}
For simplicity, let:
$$\phi(t)=\int_{s+\rho_{2M}\leq\epsilon T+M}\left(\left|\nabla^{*}u\right|^{2}+\left|\frac{\partial_{s}((s+\rho_{2M})u)}{s+\rho_{2M}}\right|^{2}\right)dx,$$
$$\psi(T)=\frac{1}{T}\int_{0}^{T}\phi(t)dt,$$
$$\gamma(t)=\frac{2c_{1}\beta E}{\eta}+\frac{1}{t\eta}\left(C_{0}E+C_{2}E\ln(1+t)+2(c_{2}+c_{3}t)flux(0,t)\right).$$
with $\eta=\sqrt{\eta_{0}}$.\newline 
Thus we want to show the existence of $T_{0}$ such that $\forall T\geq T_{0}$ we have $\psi(T)<\frac{\epsilon_{0}}{2}$.\newline
$$\psi(T)+T\frac{d\psi}{dt}=\phi(T)$$
From the differential inequality in Proposition \ref{differentialinequality} we have  
$$\phi(T)\leq\gamma(T)+\eta\psi(T)$$
so $$\frac{d\psi}{dt}+(1-\eta)\frac{\psi(T)}{T}\leq\frac{\gamma(T)}{T}$$
$$\frac{1}{T^{1-\eta}}\frac{d(T^{1-\eta}\psi)}{dt}\leq\frac{\gamma(T)}{T}$$
$$\frac{d(T^{1-\eta}\psi)}{dt}\leq\frac{\gamma(T)}{T^{\eta}}$$
$$T^{1-\eta}\psi(T)\leq\psi(1)+\int_{1}^{T}\frac{\gamma(t)}{t^{\eta}}dt$$
Hence,
\begin{equation}\label{psi0}
\psi(T)\leq J(T)+\frac{\psi(1)}{T^{1-\eta}}
\end{equation}
with $$J(T)=\frac{1}{T^{1-\eta}}\int_{1}^{T}\frac{\gamma(t)}{t^{\eta}}dt.$$
But, note that $flux(0,t)\xrightarrow[t\rightarrow\infty]{}0$ by the classical energy-conservation law on the exterior 
of a truncated cone stated above. Thus, $\exists t_{0}$ such that $\forall t\geq t_{0}$ we have 
$$\frac{1}{t\eta}\left(C_{0}E+C_{2}E\ln(1+t)+2(c_{2}+c_{3}T)flux(0,t)\right)<\frac{\epsilon_{0}(1-\eta)}{12},$$
and choose $\beta$ such that $$\frac{2c_{1}\beta E}{\eta}=\frac{\epsilon_{0}(1-\eta)}{12}.$$
Hence, $$\exists t_{0}\;\mbox{such that}\;\forall t\geq t_{0}\;\mbox{we have}\;\gamma(t)<\frac{\epsilon_{0}(1-\eta)}{6}$$
and $\gamma$ is bounded: $$\gamma(t)\leq M,\;\forall t.$$ 
Moreover, 
$$\exists T_{0}>t_{0}\;\mbox{such that}\;\forall t>T_{0},\;\mbox{we have}\;\frac{1}{t^{1-\eta}}\frac{M}{1-\eta}(t_{0}^{1-\eta}-1)<\frac{\epsilon_{0}}{6}\;\mbox{and}\;\frac{\psi(1)}{t^{1-\eta}}<\frac{\epsilon_{0}}{6}.$$
Thus, for all $T>T_{0}$,
\begin{align*}
J(T)&=\frac{1}{T^{1-\eta}}\int_{1}^{t_{0}}\frac{\gamma(t)}{t^{\eta}}dt+\frac{1}{T^{1-\eta}}\int_{t_{0}}^{T}\frac{\gamma(t)}{t^{\eta}}dt\\
&\leq\frac{1}{T^{1-\eta}}\frac{M}{1-\eta}(t_{0}^{1-\eta}-1)+\frac{1}{T^{1-\eta}}\frac{\epsilon_{0}}{6}(T^{1-\eta}-t_{0}^{1-\eta})\\
&<\frac{\epsilon_{0}}{3}
\end{align*} 
and by \eqref{psi0}, $\psi(T)<\frac{\epsilon_{0}}{2}$ which is \eqref{gronwall}.
\newline\newline
Now, from Proposition \ref{differentialinequality}, we have  
$$\int_{s+\rho_{2M}\leq T+M}\frac{u^{6}(T,x)}{3}dx\leq\eta\gamma(T)+\eta^{2}\psi(T)<\gamma(T)+\psi(T)<\epsilon_{0}$$
Hence, for all $T\geq T_{0}$, we have $$\int_{s+\rho_{2M}\leq T+M}\frac{u^{6}(T,x)}{3}dx<\epsilon_{0}$$
which ends the proof of Theorem \ref{L6decay}.
\section{Proof of the differential inequality (Proposition \ref{differentialinequality})}
The method we are going to use to prove our result is the method of multipliers and we will generalize the Morawetz multipliers that were used for star-shaped obstacles in way that suits our 
obstacle.
\subsection{Divergence Identity and Integral Equality}
 We multiply the wave equation by $$(u+\alpha\cdot\nabla u+(t+M)\partial_{t}u),$$ where $M\geq\rho_{2M}$ is the positive constant chosen in the previous section and $\alpha$ is a vector field defined as follows:
\begin{equation}\label{4'}
\alpha=(s+\rho_{2M})\nu
\end{equation}
and we get the following divergence identity:
$$\partial_{t}Q+ \div P+R=0$$
where
\begin{equation*}
\left\{
\begin{aligned}
&Q=(t+M)\frac{|\nabla u|^2}{2}+(t+M)\frac{|u|^{6}}{6}+(t+M)\frac{|\partial_{t}u|^{2}}{2}+\partial_{t}u(\alpha\cdot\nabla u)+(\partial_{t}u)u\\
&P=\left(\frac{|\nabla u|^{2}}{2}+\frac{|u|^{6}}{6}-\frac{|\partial_{t}u|^{2}}{2}\right)\alpha-\left((t+M)\partial_{t}u+\alpha\cdot\nabla u+u\right)\nabla u\\
&R=\left(\div\alpha-3\right)\frac{|\partial_{t}u|^{2}}{2}+\left(1-\div\alpha\right)\frac{|\nabla u|^{2}}{2}+\left(5-\div\alpha\right)\frac{|u|^{6}}{6}+H_{\alpha}(\nabla u,\nabla u)
\end{aligned}
\right.
\end{equation*}
where $H_{\alpha}(\nabla u,\nabla u)=\displaystyle\sum_{i,j=1}^{3}\partial_{i}\alpha_{j}\partial_{i}u\partial_{j}u=\nabla u\cdot\left(\left(\nabla u\cdot\nabla\right)\alpha\right)$.\newline\newline
Integrating the divergence identity over the truncated cone
$$K_{T_{1}}^{T_{2}}=\left\{(x,t); s+\rho_{2M}\leq t+M, T_{1}\leq t\leq T_{2}\right\},\;\;0<T_{1}<T_{2},$$ 
and applying the divergence theorem, we get
\begin{align}\label{1}
\int_{D(T_{2})}Q(T_{2},x)dx&-\int_{D(T_{1})}Q(T_{1},x)dx-\frac{1}{\sqrt{2}}\int_{M_{T_{1}}^{T_{2}}}\left(Q-P\cdot\nu\right)d\sigma\notag\\
&-\int_{T_{1}}^{T_{2}}\int_{\partial V}P\cdot nd\sigma dt+\int_{K_{T_{1}}^{T_{2}}}Rdxdt=0
\end{align}
where $d\sigma$ denotes the Lebesgue measure on the corresponding surface and $n$ is the outward pointing unit normal vector
 to $\partial V$; and where
$$D(T_{i})=\left\{x\in\Omega; s+\rho_{2M}\leq T_{i}+M\right\}$$
and
$$M_{T_{1}}^{T_{2}}=\left\{(x,t); s+\rho_{2M}=t+M, T_{1}\leq t\leq T_{2}\right\}.$$
\subsection{The differential inequality}
Now, we deal with the terms of the integral equality \eqref{1} in order to get the desired differential inequality.\newline\newline
\textbf{The boundary term}\newline\newline
By the Dirichlet boundary condition, we get:
\begin{align}\label{boundaryterm}
 -\int_{T_{1}}^{T_{2}}\int_{\partial V}P\cdot nd\sigma dt&=-\int^{T_{2}}_{T_{1}}\int_{\partial V}\frac{|\nabla u|^{2}}{2}\alpha\cdot n-(\alpha\cdot\nabla u)(\nabla u\cdot n)d\sigma dt\notag\\
&=\int^{T_{2}}_{T_{1}}\int_{\partial V}\frac{1}{2}|\nabla u\cdot n|^{2}(\alpha\cdot n)d\sigma dt\geq 0
\end{align}
since $\alpha\cdot n=(s+\rho_{2M})\nu\cdot n$ with $s+\rho_{2M}>0$ (by our assumption) and $\nu\cdot n\geq 0$ (Lemma \ref{lemma2}).\newline\newline
\textbf{The terms on the time slices}\newline\newline
We have
$$Q(T_{2},x)=(T_{2}+M)\frac{|\nabla u|^2}{2}+(T_{2}+M)\frac{|u|^{6}}{6}+(T_{2}+M)\frac{|\partial_{t}u|^{2}}{2}+\partial_{t}u(\alpha\cdot\nabla u)+(\partial_{t}u)u $$
Introduce
\begin{align*}
I(T_{2})&=\frac{1}{4}(T_{2}+M+(s+\rho_{2M}))\left[\partial_{t}u+\frac{\partial_{s}((s+\rho_{2M})u)}{{s+\rho_{2M}}}\right]^{2}+\frac{1}{4}(T_{2}+M-(s+\rho_{2M}))\left[\partial_{t}u-\frac{\partial_{s}((s+\rho_{2M})u)}{{s+\rho_{2M}}}\right]^{2}\\
&\;\;\;\;+ (T_{2}+M)\frac{u^{6}}{6}\\
&=(T_{2}+M)\left(\frac{|\partial_{t}u|^{2}}{2}+\frac{|\partial_{s}u|^{2}}{2}+\frac{u^{6}}{6}\right)+(\partial_{t}u)(s+\rho_{2M})\partial_{s}u+(\partial_{t}u)u+\frac{T_{2}+M}{2}\left(\frac{u^{2}}{(s+\rho_{2M})^{2}}+\frac{2u\partial_{s}u}{s+\rho_{2M}}\right)
\end{align*}
by Remark \ref{expressions}, we have 
$$|\nabla u|^{2}=|\partial_{s}u|^{2}+|\nabla^{*} u|^{2}$$
where $$|\nabla^{*} u|^{2}=|\nabla^{*}_{1} u|^{2}+|\nabla^{*}_{2} u|^{2}$$ and
$$|\nabla^{*}_{i} u|^{2}=\frac{1}{(\kappa_{i}s+1)^{2}|X^{0}_{\sigma_{i}}|^{2}}(\partial_{\sigma_{i}}u)^{2},\;i=1,2$$ 
so
$$I(T_{2})+(T_{2}+M)\frac{|\nabla^{*}u|^{2}}{2}=Q(T_{2},\cdot)+\frac{T_{2}+M}{2}\left(\frac{u^{2}}{(s+\rho_{2M})^{2}}+\frac{2u\partial_{s}u}{s+\rho_{2M}}\right)$$
Thus,
\begin{align*}
\int_{D(T_{2})}Q(T_{2},x)dx&=\int_{D(T_{2})}I(T_{2})+(T_{2}+M)\frac{|\nabla^{*}u|^{2}}{2}dx-\frac{T_{2}+M}{2}\int_{D(T_{2})}\frac{2u\partial_{s}u}{s+\rho_{2M}}+\frac{u^{2}}{(s+\rho_{2M})^{2}}dx\\
&=\int_{D(T_{2})}I(T_{2})+(T_{2}+M)\frac{|\nabla^{*}u|^{2}}{2}dx-\frac{T_{2}+M}{2}\int_{D(T_{2})}\frac{\partial_{s}((s+\rho_{2M})u^{2})}{(s+\rho_{2M})^{2}}dx
\end{align*}
Now, integrating by parts and using Dirichlet boundary condition, we get
\begin{equation*}
-\int_{D(T_{2})}\frac{\partial_{s}((s+\rho_{2M})u^{2})}{(s+\rho_{2M})^{2}}dx=\int_{D(T_{2})}(s+\rho_{2M})u^{2}\partial_{s}\left(\frac{(s+\rho_{1})(s+\rho_{2})}{(s+\rho_{2M})^{2}}\right)\frac{\Lambda}{\rho_{1}\rho_{2}}dsd\sigma_{1}d\sigma_{2}-\int_{\partial D(T_{2})}\frac{u^{2}}{T_{2}+M}dS_{2}
\end{equation*}
where $dS_{2}$ is the measure on $\partial D(T_{2})$ and 
$$\partial D(T_{2})=\left\{x,s+\rho_{2M}=T_{2}+M\right\}.$$
But
\begin{align*}
\partial_{s}\left(\frac{(s+\rho_{1})(s+\rho_{2})}{(s+\rho_{2M})^{2}}\right)&=\frac{(2s+\rho_{1}+\rho_{2})(s+\rho_{2M})-2(s+\rho_{1})(s+\rho_{2})}{(s+\rho_{2M})^{3}}\\
&=\frac{1}{(s+\rho_{2M})^{3}}\left(\displaystyle\sum_{i,j=1,i\neq j}^{2}(\rho_{2M}-\rho_{i})(s+\rho_{j})\right)\geq 0
\end{align*}
so
\begin{align}\label{Q2}
\int_{D(T_{2})}Q(T_{2},x)dx&=\int_{D(T_{2})}I(T_{2})+(T_{2}+M)\frac{|\nabla^{*}u|^{2}}{2}dx-\frac{T_{2}+M}{2}\int_{\partial D(T_{2})}\frac{u^{2}}{T_{2}+M}dS_{2}\notag\\
&\;\;+\frac{T_{2}+M}{2}\int_{D(T_{2})}\frac{u^{2}}{(s+\rho_{2M})^{2}}\left(\displaystyle\sum_{i,j=1,i\neq j}^{2}(\rho_{2M}-\rho_{i})(s+\rho_{j})\right)\frac{\Lambda}{\rho_{1}\rho_{2}}dsd\sigma_{1}d\sigma_{2}\notag\\
&=\int_{D(T_{2})}I(T_{2})+(T_{2}+M)\frac{|\nabla^{*}u|^{2}}{2}dx-\frac{1}{2}\int_{\partial D(T_{2})}u^{2}dS_{2}\notag\\
&\;\;+\frac{T_{2}+M}{2}\int_{D(T_{2})}\left(\displaystyle\sum_{i=1}^{2}\frac{\rho_{2M}-\rho_{i}}{s+\rho_{i}}\right)\frac{u^{2}}{(s+\rho_{2M})^{2}}dx
\end{align}
Similarly we get 
\begin{align}\label{Q1}
\int_{D(T_{1})}Q(T_{1},x)dx&=\int_{D(T_{1})}I(T_{1})+(T_{1}+M)\frac{|\nabla^{*}u|^{2}}{2}dx-\frac{1}{2}\int_{\partial D(T_{1})}u^{2}dS_{1}\notag\\
&\;\;+\frac{T_{1}+M}{2}\int_{D(T_{1})}\left(\displaystyle\sum_{i=1}^{2}\frac{\rho_{2M}-\rho_{i}}{s+\rho_{i}}\right)\frac{u^{2}}{(s+\rho_{2M})^{2}}dx
\end{align}
\textbf{The term on the mantle}\newline\newline
On the mantle $M_{T_{1}}^{T_{2}}$, we have $s+\rho_{2M}=t+M$, and recall that $\nabla u\cdot\nu=\partial_{s}u$ (Remark \ref{expressions}), thus we get:
\begin{align*}
Q-P\cdot\nu&=(s+\rho_{2M})(\partial_{t}u+\partial_{s}u)^{2}+u(\partial_{t}u+\partial_{s}u)\\
&=(s+\rho_{2M})\left(\partial_{t}u+\partial_{s}u+\frac{u}{s+\rho_{2M}}\right)^{2}-u(\partial_{t}u+\partial_{s}u)-\frac{u^{2}}{s+\rho_{2M}}
\end{align*}   
so
\begin{align}\label{mantleterm}
-\frac{1}{\sqrt{2}}\int_{M_{T_{1}}^{T_{2}}}(Q-P\cdot\nu)d\sigma&=-\frac{1}{\sqrt{2}}\int_{M_{T_{1}}^{T_{2}}}(s+\rho_{2M})\left(\partial_{t}u+\partial_{s}u+\frac{u}{s+\rho_{2M}}\right)^{2}d\sigma\notag\\
&\;\;\;\;+\frac{1}{\sqrt{2}}\int_{M_{T_{1}}^{T_{2}}}u(\partial_{t}u+\partial_{s}u)+\frac{u^{2}}{s+\rho_{2M}}d\sigma
\end{align} 
Now setting $\overline{u}(y)=u(s+\rho_{2M}-M,y)$ on the mantle, we have 
$$\nabla\overline{u}=\nu\partial_{t}u+\nabla u\;\;\mbox{and}\;\;\partial_{s}\overline{u}=\partial_{t}u+\partial_{s}u,$$ 
the second term in \eqref{mantleterm} can be written as such
\begin{align*}
\frac{1}{\sqrt{2}}\int_{M_{T_{1}}^{T_{2}}}u(\partial_{t}u+\partial_{s}u)&+\frac{u^{2}}{s+\rho_{2M}}d\sigma=\int_{\overline{M}_{T_{1}}^{T_{2}}}\left(\overline{u}\partial_{s}\overline{u}+\frac{\overline{u}^{2}}{s+\rho_{2M}}\right)dy=\int_{\overline{M}_{T_{1}}^{T_{2}}}\frac{\partial_{s}((s+\rho_{2M})^{2}\overline{u}^{2})}{2(s+\rho_{2M})^{2}}dy
\end{align*}
where $$\overline{M}_{T_{1}}^{T_{2}}=\left\{x,T_{1}+M\leq s+\rho_{2M}\leq T_{2}+M\right\}$$
Integrating by parts and using the Dirichlet boundary condition, we get
\begin{align*}
\int_{\overline{M}_{T_{1}}^{T_{2}}}\frac{\partial_{s}((s+\rho_{2M})^{2}\overline{u}^{2})}{2(s+\rho_{2M})^{2}}dy&=\int_{\overline{M}_{T_{1}}^{T_{2}}}\frac{\partial_{s}((s+\rho_{2M})^{2}\overline{u}^{2})}{2(s+\rho_{2M})^{2}}\frac{\Lambda}{\rho_{1}\rho_{2}}(s+\rho_{1})(s+\rho_{2})dsd\sigma_{1}d\sigma_{2}\\
&=-\int_{\overline{M}_{T_{1}}^{T_{2}}}\frac{1}{2}(s+\rho_{2M})^{2}\overline{u}^{2}\partial_{s}\left(\frac{(s+\rho_{1})(s+\rho_{2})}{(s+\rho_{2M})^{2}}\right)\frac{\Lambda}{\rho_{1}\rho_{2}}dsd\sigma_{1}d\sigma_{2}\\
&\;\;\;\;+\frac{1}{2}\int_{\partial D(T_{2})}\overline{u}^{2}dS_{2}-\frac{1}{2}\int_{\partial D(T_{1})}\overline{u}^{2}dS_{1}\\
&=-\frac{1}{2}\int_{\overline{M}_{T_{1}}^{T_{2}}}\left(\displaystyle\sum_{i,j=1,i\neq j}^{2}(\rho_{2M}-\rho_{i})(s+\rho_{j})\right)\frac{\overline{u}^{2}}{s+\rho_{2M}}\frac{\Lambda}{\rho_{1}\rho_{2}}dsd\sigma_{1}d\sigma_{2}\\
&\;\;\;\;+\frac{1}{2}\int_{\partial D(T_{2})}\overline{u}^{2}dS_{2}-\frac{1}{2}\int_{\partial D(T_{1})}\overline{u}^{2}dS_{1}\\
\end{align*}
Thus the term on the mantle \eqref{mantleterm} becomes
\begin{align}\label{mantletermfinal}
-\frac{1}{\sqrt{2}}\int_{M_{T_{1}}^{T_{2}}}(Q-P\cdot\nu)d\sigma&=-\frac{1}{\sqrt{2}}\int_{M_{T_{1}}^{T_{2}}}(s+{\rho_{2M}})(\partial_{t}u+\partial_{s}u+\frac{u}{s+\rho_{2M}})^{2}d\sigma\\
&\;\;\;\;-\frac{1}{2}\int_{\overline{M}_{T_{1}}^{T_{2}}}\left(\displaystyle\sum_{i,j=1,i\neq j}^{2}(\rho_{2M}-\rho_{i})(s+\rho_{j})\right)\frac{\overline{u}^{2}}{s+\rho_{2M}}\frac{\Lambda}{\rho_{1}\rho_{2}}dsd\sigma_{1}d\sigma_{2}\notag\\
&\;\;\;\;+\frac{1}{2}\int_{\partial D(T_{2})}u^{2}dS_{2}-\frac{1}{2}\int_{\partial D(T_{1})}u^{2}dS_{1}\notag
\end{align}
\textbf{The remainder term}\newline\newline
 We have $\alpha=(s+\rho_{2M})\nu$ thus by Remark \ref{expressions}, we get
\begin{align*}
\div\alpha&=\frac{1}{(\kappa_{1}s+1)(\kappa_{2}s+1)}\partial_{s}((\kappa_{1}s+1)(\kappa_{2}s+1)(s+\rho_{2M}))\\
&=1+\frac{\kappa_{1}(s+\rho_{2M})}{\kappa_{1}s+1}+\frac{\kappa_{2}(s+\rho_{2M})}{\kappa_{2}s+1}\\
&=3+\frac{\rho_{2M}-\rho_{1}}{s+\rho_{1}}+\frac{\rho_{2M}-\rho_{2}}{s+\rho_{2}}
\end{align*}
since $s+\rho_{i}>0$ and $\rho_{2M}\geq\rho_{2}\geq\rho_{1}$ then $\div\alpha-3\geq 0$ and
\begin{equation}\label{R1}
\int_{K_{T_{1}}^{T_{2}}}\left(\div\alpha-3\right)\frac{|\partial_{t}u|^{2}}{2}dxdt\geq 0
\end{equation}
Now, in the following remainder term
\begin{equation*}
\int_{K_{T_{1}}^{T_{2}}}\left(5-\div\alpha\right)\frac{|u|^{6}}{6}dxdt
\end{equation*}
we have
\begin{align*}
5-\div\alpha&=2-\frac{\rho_{2M}-\rho_{1}}{s+\rho_{1}}-\frac{\rho_{2M}-\rho_{2}}{s+\rho_{2}}
\end{align*}
Imposing the following geometric condition
\begin{equation*}
\min_{\partial V}(s_{0}+\rho_{1}-(\rho_{2M}-\rho_{1}))>0
\end{equation*}
we get that (recall that $s\geq s_{0}$) 
$$\frac{\rho_{2M}-\rho_{1}}{s+\rho_{1}}+\frac{\rho_{2M}-\rho_{2}}{s+\rho_{2}}<2$$
and thus
\begin{equation}\label{R2}
\int_{K_{T_{1}}^{T_{2}}}\left(5-\div\alpha\right)\frac{|u|^{6}}{6}dxdt\geq 0
\end{equation}
We still have to deal with the following term in $R$:
$$H_{\alpha}(\nabla u,\nabla u)+\left(1-\div\alpha\right)\frac{|\nabla u|^{2}}{2}$$
We have $H_{\alpha}(\nabla u,\nabla u)=\nabla u\cdot\left(\left(\nabla u\cdot\nabla\right)\alpha\right)$; and by Remark \ref{expressions}, we have
$$\nabla u=\partial_{s}u\nu+\frac{1}{(\kappa_{1}s+1)\left|X^{0}_{\sigma_{1}}\right|^{2}}\partial_{\sigma_{1}}uX^{0}_{\sigma_{1}}+\frac{1}{(\kappa_{2}s+1)\left|X^{0}_{\sigma_{2}}\right|^{2}}\partial_{\sigma_{2}}uX^{0}_{\sigma_{2}}$$
$$\nabla u \cdot\nabla=\partial_{s} u\partial_{s}+\frac{1}{(\kappa_{1}s+1)^{2}|X^{0}_{\sigma_{1}}|^{2}}\partial_{\sigma_{1}}u\partial_{\sigma_{1}}+\frac{1}{(\kappa_{2}s+1)^{2}|X^{0}_{\sigma_{2}}|^{2}}\partial_{\sigma_{2}}u\partial_{\sigma_{2}}$$
\begin{align*}
(\nabla u \cdot\nabla)\alpha=&\left(\left(\nabla u \cdot\nabla \right)\left(s+\rho_{2M}\right)\right)\nu+\left(s+\rho_{2M}\right)\left(\left(\nabla u \cdot\nabla \right)\nu\right)\\
 =&\partial_{s} u\nu+(s+\rho_{2M})\left(\frac{\kappa_{1}}{(\kappa_{1}s+1)^{2}|X^{0}_{\sigma_{1}}|^{2}}\partial_{\sigma_{1}}uX^{0}_{\sigma_{1}}+\frac{\kappa_{2}}{(\kappa_{2}s+1)^{2}|X^{0}_{\sigma_{2}}|^{2}}\partial_{\sigma_{2}}uX^{0}_{\sigma_{2}}\right)
\end{align*}
Hence
\begin{align*}
H_{\alpha}(\nabla u,\nabla u)=\nabla u\cdot\left(\left(\nabla u\cdot\nabla\right)\alpha\right)=\left(\partial_{s}u\right)^{2}+\frac{\kappa_{1}(s+\rho_{2M})}{(\kappa_{1}s+1)^{3}|X^{0}_{\sigma_{1}}|^{2}}(\partial_{\sigma_{1}}u)^{2}+\frac{\kappa_{2}(s+\rho_{2M})}{(\kappa_{2}s+1)^{3}|X^{0}_{\sigma_{2}}|^{2}}(\partial_{\sigma_{2}}u)^{2}
\end{align*}
Using 
\begin{align*}
|\nabla u|^{2}&=(\partial_{s}u)^{2}+\frac{1}{(\kappa_{1}s+1)^{2}|X^{0}_{\sigma_{1}}|^{2}}(\partial_{\sigma_{1}}u)^{2}+\frac{1}{(\kappa_{2}s+1)^{2}|X^{0}_{\sigma_{2}}|^{2}}(\partial_{\sigma_{2}}u)^{2}\\
&=(\partial_{s}u)^{2}+\left|\nabla_{1}^{*}u\right|^{2}+\left|\nabla_{2}^{*}u\right|^{2}
\end{align*}
we get
\begin{align*}
 H_{\alpha}(\nabla u,\nabla u)&=|\nabla u|^{2}+\frac{\rho_{2M}-\rho_{1}}{s+\rho_{1}}|\nabla^{*}_{1}u|^{2}+\frac{\rho_{2M}-\rho_{2}}{s+\rho_{2}}|\nabla^{*}_{2}u|^{2}\geq 0
\end{align*}
Moreover, we have 
$$\left(1-\div\alpha\right)\frac{|\nabla u|^{2}}{2}=-|\nabla u|^{2}-\left(\frac{\rho_{2M}-\rho_{1}}{s+\rho_{1}}+\frac{\rho_{2M}-\rho_{2}}{s+\rho_{2}}\right)\frac{|\nabla u|^{2}}{2}$$
Hence
\begin{align}\label{10}
\int_{K_{T_{1}}^{T_{2}}}&H_{\alpha}(\nabla u,\nabla u)+(1-\div\alpha)\frac{|\nabla u|^{2}}{2}dxdt\\
&=\int_{K_{T_{1}}^{T_{2}}}\left(\displaystyle\sum_{i=1}^{2}\frac{\rho_{2M}-\rho_{i}}{s+\rho_{i}}|\nabla_{i}^{*}u|^{2}\right)dxdt-\int_{K_{T_{1}}^{T_{2}}}\left(\displaystyle\sum_{i=1}^{2}\frac{\rho_{2M}-\rho_{i}}{s+\rho_{i}}\right)\frac{|\nabla u|^{2}}{2}dxdt\notag
\end{align}
Recall that $$|\nabla u|^{2}=|\partial_{s}u|^{2}+|\nabla^{*}u|^{2}$$
and note that $$\frac{\partial_{s}((s+\rho_{2M})u)}{s+\rho_{2M}}=\partial_{s}u+\frac{u}{s+\rho_{2M}},$$ hence
$$|\nabla u|^{2}=|\nabla^{*}u|^{2}+\left|\frac{\partial_{s}((s+\rho_{2M})u)}{s+\rho_{2M}}\right|^{2}-\left|\frac{u}{s+\rho_{2M}}\right|^{2}-\frac{2u\partial_{s}u}{s+\rho_{2M}}$$
Substituting this in \eqref{10}, we get:
\begin{align}\label{10aftersub}
&\int_{K_{T_{1}}^{T_{2}}}H_{\alpha}(\nabla u,\nabla u)+(1-\div\alpha)\frac{|\nabla u|^{2}}{2}dxdt\notag\\
&\;\;=\int_{K_{T_{1}}^{T_{2}}}\left(\displaystyle\sum_{i=1}^{2}\frac{\rho_{2M}-\rho_{i}}{s+\rho_{i}}|\nabla_{i}^{*}u|^{2}-\frac{1}{2}\left(\displaystyle\sum_{i=1}^{2}\frac{\rho_{2M}-\rho_{i}}{s+\rho_{i}}\right)|\nabla^{*}u|^{2}\right)dxdt\notag\\
&\;\;\;-\frac{1}{2}\int_{K_{T_{1}}^{T_{2}}}\left(\displaystyle\sum_{i=1}^{2}\frac{\rho_{2M}-\rho_{i}}{s+\rho_{i}}\right)\left|\frac{\partial_{s}((s+\rho_{2M})u)}{s+\rho_{2M}}\right|^{2}dxdt\\
&\;\;\;+\frac{1}{2}\int_{K_{T_{1}}^{T_{2}}}\left(\displaystyle\sum_{i=1}^{2}\frac{\rho_{2M}-\rho_{i}}{s+\rho_{i}}\right)\left|\frac{u}{s+\rho_{2M}}\right|^{2}dxdt\notag\\
&\;\;\;+\frac{1}{2}\int_{K_{T_{1}}^{T_{2}}}\left(\displaystyle\sum_{i=1}^{2}\frac{\rho_{2M}-\rho_{i}}{s+\rho_{i}}\right)\frac{2u\partial_{s}u}{s+\rho_{2M}}dxdt\notag\\
&\;\;=\Rmnum{1}+\Rmnum{2}+\Rmnum{3}+\Rmnum{4}\notag
\end{align}
\begin{align}\label{Rmnum1}
\Rmnum{1}=\frac{1}{2}\int_{K_{T_{1}}^{T_{2}}}\left(\frac{\rho_{2M}-\rho_{1}}{s+\rho_{1}}-\frac{\rho_{2M}-\rho_{2}}{s+\rho_{2}}\right)|\nabla^{*}_{1}u|^{2}+\left(\frac{\rho_{2M}-\rho_{2}}{s+\rho_{2}}-\frac{\rho_{2M}-\rho_{1}}{s+\rho_{1}}\right)|\nabla^{*}_{2}u|^{2}dxdt
\end{align}
Integrating by parts the last term and using the Dirichlet boundary condition, we get:
\begin{align*}
\Rmnum{4}&=\displaystyle\sum_{i=1}^{2}\frac{1}{2}\int_{K_{T_{1}}^{T_{2}}}\frac{\rho_{2M}-\rho_{i}}{s+\rho_{i}}\frac{2u\partial_{s}u}{s+\rho_{2M}}dxdt\\
&=\displaystyle\sum_{i,j=1,i\neq j}^{2}\frac{1}{2}\int^{T_{2}}_{T_{1}}\int_{s+\rho_{2M}\leq t+M}(\rho_{2M}-\rho_{i})\partial_{s}(u^{2})\frac{s+\rho_{j}}{s+\rho_{2M}}\frac{\Lambda}{\rho_{1}\rho_{2}}dsd\sigma_{1}d\sigma_{2}dt\\
&=-\displaystyle\sum_{i,j=1,i\neq j}^{2}\frac{1}{2}\int^{T_{2}}_{T_{1}}\int_{s+\rho_{2M}\leq t+M}(\rho_{2M}-\rho_{i})u^{2}\partial_{s}\left(\frac{s+\rho_{j}}{s+\rho_{2M}}\right)\frac{\Lambda}{\rho_{1}\rho_{2}}dsd\sigma_{1}d\sigma_{2}dt\\
&\;\;\;+\displaystyle\sum_{i,j=1,i\neq j}^{2}\frac{1}{2}\int_{\overline{M}^{T_{2}}_{T_{1}}}(\rho_{2M}-\rho_{i})u^{2}\frac{s+\rho_{j}}{s+\rho_{2M}}\frac{\Lambda}{\rho_{1}\rho_{2}}dsd\sigma_{1}d\sigma_{2}\\
&=-\displaystyle\sum_{i,j=1,i\neq j}^{2}\frac{1}{2}\int^{T_{2}}_{T_{1}}\int_{s+\rho_{2M}\leq t+M}(\rho_{2M}-\rho_{i})(\rho_{2M}-\rho_{j})\frac{u^{2}}{(s+\rho_{2M})^{2}}\frac{\Lambda}{\rho_{1}\rho_{2}}dsd\sigma_{1}d\sigma_{2}dt\\
&\;\;\;+\displaystyle\sum_{i,j=1,i\neq j}^{2}\frac{1}{2}\int_{\overline{M}^{T_{2}}_{T_{1}}}(\rho_{2M}-\rho_{i})(s+\rho_{j})\frac{u^{2}}{s+\rho_{2M}}\frac{\Lambda}{\rho_{1}\rho_{2}}dsd\sigma_{1}d\sigma_{2}
\end{align*} 
and we have 
\begin{align*}
\Rmnum{3}=\displaystyle\sum_{i,j=1,i\neq j}^{2}\frac{1}{2}\int^{T_{2}}_{T_{1}}\int_{s+\rho_{2M}\leq t+M}(\rho_{2M}-\rho_{i})(s+\rho_{j})\frac{u^{2}}{(s+\rho_{2M})^{2}}\frac{\Lambda}{\rho_{1}\rho_{2}}dsd\sigma_{1}d\sigma_{2}dt
\end{align*}
Hence,
\begin{align*}
\Rmnum{3}+\Rmnum{4}&=\frac{1}{2}\int_{\overline{M}^{T_{2}}_{T_{1}}}\displaystyle\sum_{i,j=1,i\neq j}^{2}(\rho_{2M}-\rho_{i})(s+\rho_{j})\frac{u^{2}}{s+\rho_{2M}}\frac{\Lambda}{\rho_{1}\rho_{2}}dsd\sigma_{1}d\sigma_{2}\\
&\;\;\;+\frac{1}{2}\int^{T_{2}}_{T_{1}}\int_{s+\rho_{2M}\leq t+M}\displaystyle\sum_{i,j=1,i\neq j}^{2}(\rho_{2M}-\rho_{i})(s+2\rho_{j}-\rho_{2M})\frac{u^{2}}{(s+\rho_{2M})^{2}}\frac{\Lambda}{\rho_{1}\rho_{2}}dsd\sigma_{1}d\sigma_{2}dt
\end{align*}
Now substituting in \eqref{10aftersub} we get
\begin{align}\label{Hexpression}
\int_{K_{T_{1}}^{T_{2}}}&H_{\alpha}(\nabla u,\nabla u)+(1-\div\alpha)\frac{|\nabla u|^{2}}{2}dxdt\notag\\
&=\frac{1}{2}\int_{K_{T_{1}}^{T_{2}}}\left(\frac{\rho_{2M}-\rho_{1}}{s+\rho_{1}}-\frac{\rho_{2M}-\rho_{2}}{s+\rho_{2}}\right)|\nabla^{*}_{1}u|^{2}+\left(\frac{\rho_{2M}-\rho_{2}}{s+\rho_{2}}-\frac{\rho_{2M}-\rho_{1}}{s+\rho_{1}}\right)|\nabla^{*}_{2}u|^{2}dxdt\notag\\
&\;-\frac{1}{2}\int_{K_{T_{1}}^{T_{2}}}\left(\displaystyle\sum_{i=1}^{2}\frac{\rho_{2M}-\rho_{i}}{s+\rho_{i}}\right)\left|\frac{\partial_{s}((s+\rho_{2M})u)}{s+\rho_{2M}}\right|^{2}dxdt\\
&\;+\frac{1}{2}\int_{\overline{M}^{T_{2}}_{T_{1}}}\displaystyle\sum_{i,j=1,i\neq j}^{2}(\rho_{2M}-\rho_{i})(s+\rho_{j})\frac{u^{2}}{s+\rho_{2M}}\frac{\Lambda}{\rho_{1}\rho_{2}}dsd\sigma_{1}d\sigma_{2}\notag\\
&\;+\frac{1}{2}\int^{T_{2}}_{T_{1}}\int_{s+\rho_{2M}\leq t+M}\displaystyle\sum_{i,j=1,i\neq j}^{2}(\rho_{2M}-\rho_{i})(s+2\rho_{j}-\rho_{2M})\frac{u^{2}}{(s+\rho_{2M})^{2}}\frac{\Lambda}{\rho_{1}\rho_{2}}dsd\sigma_{1}d\sigma_{2}dt\notag
\end{align}
and note that since we imposed the geometric condition
$$\min_{\partial V}(s_{0}+\rho_{1}-(\rho_{2M}-\rho_{1}))>0$$
we have $s+2\rho_{j}-\rho_{2M}>0$ and thus the last term in \eqref{Hexpression} is nonnegative.\newline\newline
\textbf{The differential inequality}\newline\newline
Now, summing up all the terms (\eqref{boundaryterm}, \eqref{Q2}, \eqref{Q1}, \eqref{mantletermfinal}, \eqref{R1}, \eqref{R2}, and \eqref{Hexpression}) in the integral equality \eqref{1} and dropping the nonnegative terms, we get
the following differential inequality:
\begin{align*}
\int_{D(T_{2})}&I(T_{2})+(T_{2}+M)\frac{\left|\nabla^{*}u\right|^{2}}{2}dx\\
&\;\leq\int_{D(T_{1})}I(T_{1})+(T_{1}+M)\frac{\left|\nabla^{*}u\right|^{2}}{2}dx\\
&\;\;\;+\frac{T_{1}+M}{2}\int_{D(T_{1})}\left(\displaystyle\sum_{i=1}^{2}\frac{\rho_{2M}-\rho_{i}}{s+\rho_{i}}\right)\frac{u^{2}}{(s+\rho_{2M})^{2}}dx\\
&\;\;\;+\frac{1}{\sqrt{2}}\int_{M_{T_{1}}^{T_{2}}}(s+\rho_{2M})\left(\partial_{t}u+\partial_{s}u+\frac{u}{s+\rho_{2M}}\right)^{2}d\sigma\\
&\;\;\;+\frac{1}{2}\int_{K_{T_{1}}^{T_{2}}}\left(\frac{\rho_{2M}-\rho_{2}}{s+\rho_{2}}-\frac{\rho_{2M}-\rho_{1}}{s+\rho_{1}}\right)|\nabla^{*}_{1}u|^{2}+\left(\frac{\rho_{2M}-\rho_{1}}{s+\rho_{1}}-\frac{\rho_{2M}-\rho_{2}}{s+\rho_{2}}\right)|\nabla^{*}_{2}u|^{2}dxdt\notag\\
&\;\;\;+\frac{1}{2}\int_{K_{T_{1}}^{T_{2}}}\left(\displaystyle\sum_{i=1}^{2}\frac{\rho_{2M}-\rho_{i}}{s+\rho_{i}}\right)\left|\frac{\partial_{s}((s+\rho_{2M})u)}{s+\rho_{2M}}\right|^{2}dxdt\\
\end{align*}
Recall that for $i=1,2$ we have: 
\begin{align*}
 I(T_{i})&=\frac{1}{4}(T_{i}+M+(s+\rho_{2M}))\left[\partial_{t}u+\frac{\partial_{s}((s+\rho_{2M})u)}{s+\rho_{2M}}\right]^{2}+\frac{1}{4}(T_{i}+M-(s+\rho_{2M}))\left[\partial_{t}u-\frac{\partial_{s}((s+\rho_{2M})u)}{s+\rho_{2M}}\right]^{2}\\
&\;\;\;\;+(T_{i}+M)\frac{u^{6}}{6}
\end{align*}
Thus
$$I(T_{2})\geq\frac{1}{2}(T_{2}+M-(s+\rho_{2M}))\left((\partial_{t}u)^{2}+\left|\frac{\partial_{s}((s+\rho_{2M})u)}{s+\rho_{2M}}\right|^{2}\right)+(T_{2}+M)\frac{u^{6}}{6}$$
we also have,
$$I(T_{1})\leq\frac{1}{2}(T_{1}+M+(s+\rho_{2M}))\left((\partial_{t}u)^{2}+\left|\frac{\partial_{s}((s+\rho_{2M})u)}{s+\rho_{2M}}\right|^{2}\right)+(T_{1}+M)\frac{u^{6}}{6}$$
so we get
\begin{align}\label{control}
&\int_{D(T_{2})}\frac{1}{2}(T_{2}+M-(s+\rho_{2M}))\left((\partial_{t}u)^{2}+\left|\frac{\partial_{s}((s+\rho_{2M})u)}{s+\rho_{2M}}\right|^{2}\right)+(T_{2}+M)\frac{u^{6}}{6}+(T_{2}+M)\frac{\left|\nabla^{*}u\right|^{2}}{2}dx\notag\\
&\;\leq\int_{D(T_{1})}\frac{1}{2}(T_{1}+M+(s+\rho_{2M}))\left((\partial_{t}u)^{2}+\left|\frac{\partial_{s}((s+\rho_{2M})u)}{s+\rho_{2M}}\right|^{2}+\frac{u^{6}}{3}+\left|\nabla^{*}u\right|^{2}\right)dx\\
&\;\;+\frac{T_{1}+M}{2}\int_{D(T_{1})}\left(\displaystyle\sum_{i=1}^{2}\frac{\rho_{2M}-\rho_{i}}{s+\rho_{i}}\right)\frac{u^{2}}{(s+\rho_{2M})^{2}}dx+\frac{1}{\sqrt{2}}\int_{M_{T_{1}}^{T_{2}}}(s+\rho_{2M})\left(\partial_{t}u+\partial_{s}u+\frac{u}{s+\rho_{2M}}\right)^{2}d\sigma\notag\\
&\;\;+\frac{1}{2}\int_{K_{T_{1}}^{T_{2}}}\left(\frac{\rho_{2M}-\rho_{2}}{s+\rho_{2}}-\frac{\rho_{2M}-\rho_{1}}{s+\rho_{1}}\right)|\nabla^{*}_{1}u|^{2}+\left(\frac{\rho_{2M}-\rho_{1}}{s+\rho_{1}}-\frac{\rho_{2M}-\rho_{2}}{s+\rho_{2}}\right)|\nabla^{*}_{2}u|^{2}dxdt\notag\\
&\;\;+\frac{1}{2}\int_{K_{T_{1}}^{T_{2}}}\left(\displaystyle\sum_{i=1}^{2}\frac{\rho_{2M}-\rho_{i}}{s+\rho_{i}}\right)\left|\frac{\partial_{s}((s+\rho_{2M})u)}{s+\rho_{2M}}\right|^{2}dxdt\notag\\
&\;\leq A_{1}+A_{2}+A_{3}+A_{4}+A_{5}\notag
\end{align}

We have 
\begin{align*}
\left|\frac{\partial_{s}((s+\rho_{2M})u)}{s+\rho_{2M}}\right|^{2}=\left|\frac{u}{s+\rho_{2M}}+\partial_{s}u\right|^{2}\leq2\left|\frac{u}{s+\rho_{2M}}\right|^{2}+2|\partial_{s}u|^{2}\leq2\left|\frac{u}{s+\rho_{2M}}\right|^{2}+2|\nabla u|^{2}
\end{align*}
and on $D(T_{1})$ we have: $s+\rho_{2M}\leq T_{1}+M$, thus
$$A_{1}\leq (T_{1}+M)\int_{D(T_{1})}\left((\partial_{t}u)^{2}+2\left|\frac{u}{s+\rho_{2M}}\right|^{2}+3|\nabla u|^{2}+\frac{u^{6}}{3}\right)dx$$
Moreover, by the geometric condition we imposed 
$$\min_{\partial V}(s_{0}+\rho_{1}-(\rho_{2M}-\rho_{1}))>0$$
we have 
$$\displaystyle\sum_{i=1}^{2}\frac{\rho_{2M}-\rho_{i}}{s+\rho_{i}}<2$$
Thus $$A_{2}\leq (T_{1}+M)\int_{D(T_{1})}\left|\frac{u}{s+\rho_{2M}}\right|^{2}dx$$
so $$A_{1}+A_{2}\leq(T_{1}+M)\int_{D(T_{1})}\left((\partial_{t}u)^{2}+3\left|\frac{u}{s+\rho_{2M}}\right|^{2}+3|\nabla u|^{2}+\frac{u^{6}}{3}\right)dx$$
Recall that there exist a constant $a_{0}>0$ such that $s+\rho_{2M}\geq a_{0}r$ (Lemma \ref{a0}). Thus,
$$\int\left|\frac{u}{s+\rho_{2M}}\right|^{2}dx\leq\frac{1}{a_{0}^{2}}\int\frac{u^{2}}{r^{2}}dx$$
and by Hardy's inequality:
$$\int\frac{u^{2}}{r^{2}}dx\leq C\int|\nabla u|^{2}dx$$
we get that
\begin{equation}\label{hardy}
\int\left|\frac{u}{s+\rho_{2M}}\right|^{2}dx\lesssim\int|\nabla u|^{2}dx
\end{equation}
so $$A_{1}+A_{2}\leq (c_{0}+c_{1}T_{1})E$$
where $c_{0}$ and $c_{1}$ are constants that depend on the geometry of the obstacle and $E$ is the conserved energy. Now, the term on the mantle $A_{3}$ can be written as follows:
\begin{align*}
A_{3}&=\int_{\overline{M}_{T_{1}}^{T_{2}}}(s+\rho_{2M})\left(\partial_{s}\overline{u}+\frac{\overline{u}}{s+\rho_{2M}}\right)^{2}dy\\
&\leq(T_{2}+M)\int_{\overline{M}_{T_{1}}^{T_{2}}}\left(\partial_{s}\overline{u}+\frac{\overline{u}}{s+\rho_{2M}}\right)^{2}dy\\
&\leq 2(T_{2}+M)\int_{\overline{M}_{T_{1}}^{T_{2}}}\left(|\partial_{s}\overline{u}|^{2}+\left|\frac{\overline{u}}{s+\rho_{2M}}\right|^{2}\right)dy
\end{align*}
Similarly to \eqref{hardy}, we have
$$\int\left|\frac{\overline{u}}{s+\rho_{2M}}\right|^{2}dy\lesssim\int|\nabla \overline{u}|^{2}dy$$
hence 
$$A_{3}\lesssim(T_{2}+M)\int_{\overline{M}_{T_{1}}^{T_{2}}}|\nabla\overline{u}|^{2}dy\leq (c_{2}+c_{3}T_{2})flux(T_{1},T_{2})$$
where $c_{2}$ and $c_{3}$ are constants that depend on the geometry of the obstacle, and where
$$flux(T_{1},T_{2})=\int_{M_{T_{1}}^{T_{2}}}\left(\frac{1}{2}\left|\nu\partial_{t}u+\nabla u\right|^{2}+\frac{u^{6}}{6}\right)d\sigma=\sqrt{2}\int_{\overline{M}_{T_{1}}^{T_{2}}}\left(\frac{|\nabla\overline{u}|^{2}}{2}+\frac{\overline{u}^{6}}{6}\right)dy.$$
\begin{align*}
A_{4}&\leq\frac{1}{2}\int_{K_{T_{1}}^{T_{2}}}\left|\frac{\rho_{2M}-\rho_{2}}{s+\rho_{2}}-\frac{\rho_{2M}-\rho_{1}}{s+\rho_{1}}\right||\nabla^{*}u|^{2}dxdt\\
&\leq\frac{1}{2}\int_{K_{T_{1}}^{T_{2}}}\left(\frac{\rho_{2M}-\rho_{2}}{s+\rho_{2}}+\frac{\rho_{2M}-\rho_{1}}{s+\rho_{1}}\right)|\nabla^{*}u|^{2}dxdt
\end{align*}
thus,
$$A_{4}+A_{5}\leq\frac{1}{2}\int_{K_{T_{1}}^{T_{2}}}\left(\frac{\rho_{2M}-\rho_{2}}{s+\rho_{2}}+\frac{\rho_{2M}-\rho_{1}}{s+\rho_{1}}\right)\left(|\nabla^{*}u|^{2}+\left|\frac{\partial_{s}((s+\rho_{2M})u)}{s+\rho_{2M}}\right|^{2}\right)dxdt$$
Thus \eqref{control} becomes
\begin{align}\label{control1}
&\int_{D(T_{2})}(T_{2}+M-(s+\rho_{2M}))\left(\left|\nabla^{*}u\right|^{2}+\left|\frac{\partial_{s}((s+\rho_{2M})u)}{s+\rho_{2M}}\right|^{2}\right)+T_{2}\frac{u^{6}}{3}dx\\
&\;\leq 2(c_{0}+c_{1}T_{1})E+2(c_{2}+c_{3}T_{2})flux(T_{1},T_{2})+\sum_{i=1}^{2}\int_{T_{1}}^{T_{2}}\int_{\Omega}\frac{\rho_{2M}-\rho_{i}}{s+\rho_{i}}\left(|\nabla^{*}u|^{2}+\left|\frac{\partial_{s}((s+\rho_{2M})u)}{s+\rho_{2M}}\right|^{2}\right)dxdt\notag
\end{align}

Split $\Omega$ into $\left\{s+\rho_{2M}\leq\epsilon t+M\right\}$ and $\left\{s+\rho_{2M}> \epsilon t+M\right\}$, 
where $0<\epsilon<1$ is a constant that depends on the geometry of the obstacle to be later specified, then the last term in \eqref{control1} becomes 
\begin{align*}
J&=\sum_{i=1}^{2}\int_{T_{1}}^{T_{2}}\int_{\Omega}\frac{\rho_{2M}-\rho_{i}}{s+\rho_{i}}\left(|\nabla^{*}u|^{2}+\left|\frac{\partial_{s}((s+\rho_{2M})u)}{s+\rho_{2M}}\right|^{2}\right)dxdt\\
&\leq\sum_{i=1}^{2}\int_{T_{1}}^{T_{2}}\int_{s+\rho_{2M}>\epsilon t+M}\frac{\rho_{2M}-\rho_{i}}{s+\rho_{2M}-(\rho_{2M}-\rho_{i})}\left(|\nabla^{*}u|^{2}+\left|\frac{\partial_{s}((s+\rho_{2M})u)}{s+\rho_{2M}}\right|^{2}\right)dxdt\\
&\;\;+\max_{\partial V}\left(\frac{\rho_{2M}-\rho_{1}}{s_{0}+\rho_{1}}+\frac{\rho_{2M}-\rho_{2}}{s_{0}+\rho_{2}}\right)\int_{T_{1}}^{T_{2}}\int_{s+\rho_{2M}\leq\epsilon t+M}\left(|\nabla^{*}u|^{2}+\left|\frac{\partial_{s}((s+\rho_{2M})u)}{s+\rho_{2M}}\right|^{2}\right)dxdt\\
&=J_{1}+J_{2}
\end{align*}
\begin{align*}
J_{1}&\leq\sum_{i=1}^{2}\int_{T_{1}}^{T_{2}}\frac{\rho_{2M}-\rho_{i}}{\epsilon t+N}\left(\int_{s+\rho_{2M}>\epsilon t+M}\left(|\nabla^{*}u|^{2}+\left|\frac{\partial_{s}((s+\rho_{2M})u)}{s+\rho_{2M}}\right|^{2}\right)dx\right)dt\\
&\leq 2\int_{T_{1}}^{T_{2}}\frac{\rho_{2M}-\rho_{1m}}{\epsilon t+N}\left(\int_{s+\rho_{2M}>\epsilon t+M}\left(|\nabla^{*}u|^{2}+\left|\frac{\partial_{s}((s+\rho_{2M})u)}{s+\rho_{2M}}\right|^{2}\right)dx\right)dt
\end{align*}
with $N=M-\rho_{2M}\geq0$. Since $s+\rho_{2M}\geq a_{0}r$ and using Hardy's inequality, we get
$$\int_{s+\rho_{2M}>\epsilon t+M}\left(|\nabla^{*}u|^{2}+\left|\frac{\partial_{s}((s+\rho_{2M})u)}{s+\rho_{2M}}\right|^{2}\right)dx\lesssim E$$
hence,
$$J_{1}\lesssim\frac{2(\rho_{2M}-\rho_{1m})}{\epsilon}\ln\left(\frac{\epsilon T_{2}+N}{\epsilon T_{1}+N}\right)E\leq C_{1}E+C_{2}E\ln(1+T_{2})$$
where $C_{1}$ and $C_{2}$ are constants that depend on the geometry of the obstacle. Thus 
\begin{equation}\label{J}
 J\leq C_{1}E+C_{2}E\ln(1+T_{2})+\eta_{0}\int_{T_{1}}^{T_{2}}\int_{s+\rho_{2M}\leq\epsilon t+M}\left(|\nabla^{*}u|^{2}+\left|\frac{\partial_{s}((s+\rho_{2M})u)}{s+\rho_{2M}}\right|^{2}\right)dxdt
\end{equation}
with $$\eta_{0}=\max_{\partial V}\left(\frac{\rho_{2M}-\rho_{1}}{s_{0}+\rho_{1}}+\frac{\rho_{2M}-\rho_{2}}{s_{0}+\rho_{2}}\right)$$
The geometric condition $$\min_{\partial V}(s_{0}+\rho_{1}-(\rho_{2M}-\rho_{1}))>0$$ which we assumed so far implies that $\eta_{0}<2$ which is not
enough as we want $0<\eta_{0}<1$ for the proof of the $L^{6}$ decay estimate (Theorem \ref{L6decay}).
Thus, we impose a stronger condition 
$$\min_{\partial V}(s_{0}+\rho_{1}-2(\rho_{2M}-\rho_{1}))>0$$
On the other hand, 
\begin{align}\label{lhs}
&\int_{D(T_{2})}(T_{2}+M-(s+\rho_{2M}))\left(\left|\nabla^{*}u\right|^{2}+\left|\frac{\partial_{s}((s+\rho_{2M})u)}{s+\rho_{2M}}\right|^{2}\right)+T_{2}\frac{u^{6}}{3}dx\\
&\;\;\geq T_{2}(1-\epsilon)\int_{s+\rho_{2M}\leq\epsilon T_{2}+M}\left(\left|\nabla^{*}u\right|^{2}+\left|\frac{\partial_{s}((s+\rho_{2M})u)}{s+\rho_{2M}}\right|^{2}\right)dx+T_{2}\int_{D(T_{2})}\frac{u^{6}}{3}dx\notag
\end{align}
Using \eqref{J} and \eqref{lhs}, \eqref{control1} becomes
\begin{align}\label{control2}
&T_{2}(1-\epsilon)\int_{s+\rho_{2M}\leq\epsilon T_{2}+M}\left(\left|\nabla^{*}u\right|^{2}+\left|\frac{\partial_{s}((s+\rho_{2M})u)}{s+\rho_{2M}}\right|^{2}\right)dx+T_{2}\int_{D(T_{2})}\frac{u^{6}}{3}dx\notag\\
&\;\leq C_{0}E+2c_{1}T_{1}E+C_{2}E\ln(1+T_{2})+2(c_{2}+c_{3}T_{2})flux(T_{1},T_{2})\\
&\;\;\;+\eta_{0}\int_{T_{1}}^{T_{2}}\int_{s+\rho_{2M}\leq\epsilon t+M}\left(|\nabla^{*}u|^{2}+\left|\frac{\partial_{s}((s+\rho_{2M})u)}{s+\rho_{2M}}\right|^{2}\right)dxdt\notag
\end{align}
Now, setting $T_{2}=T$ and $T_{1}=\beta T$ for some $0<\beta<1$ and choosing $0<\epsilon<1$ such that $\epsilon=1-\sqrt{\eta_{0}}$, \eqref{control2} yields
\begin{align}\label{control3}
&T\sqrt{\eta_{0}}\int_{s+\rho_{2M}\leq\epsilon T+M}\left(\left|\nabla^{*}u\right|^{2}+\left|\frac{\partial_{s}((s+\rho_{2M})u)}{s+\rho_{2M}}\right|^{2}\right)dx+T\int_{s+\rho_{2M}\leq T+M}\frac{u^{6}(T,x)}{3}dx\notag\\
&\;\leq C_{0}E+2c_{1}\beta TE+C_{2}E\ln(1+T)+2(c_{2}+c_{3}T)flux(0,T)\\
&\;\;\;+\eta_{0}\int_{0}^{T}\int_{s+\rho_{2M}\leq\epsilon t+M}\left(|\nabla^{*}u|^{2}+\left|\frac{\partial_{s}((s+\rho_{2M})u)}{s+\rho_{2M}}\right|^{2}\right)dxdt\notag
\end{align}
which ends the proof of Proposition \ref{differentialinequality}.

\section{Proof of the scattering (Corollary \ref{scattering})}
As a result of the $L^{6}$ decay estimate we get the scattering result of Corollary \ref{scattering}. The proof of this corollary 
was done in the paper of Blair, Smith, and Sogge \cite{BSS09} and we replicate it here for the sake of completeness.\newline
We have the following Strichartz estimate on functions $w(t,x)$ satisfying homogeneous Dirichlet boundary condition on 
non-trapping obstacles
\begin{equation}\label{strichartz}
\left\|w\right\|_{L^{5}(\mathbb{R};L^{10}(\Omega))}+\left\|w\right\|_{L^{4}(\mathbb{R};L^{12}(\Omega))}\leq c\left(\left\|\left(\nabla w(0,\cdot),\partial_{t}w(0,\cdot)\right)\right\|_{L^{2}(\Omega)}+\left\|\Box w\right\|_{L^{1}(\mathbb{R};L^{2}(\Omega))}\right)
\end{equation} 
and we define the conserved energy of the linear equation \eqref{linearequation}
$$E_{0}(v;t)=\frac{1}{2}\int_{\Omega}|\nabla v|^{2}+|\partial_{t}v|^{2}dx$$
Attention was restricted to the $v_{+}$ function, as symmetric arguments will yield the existence of a $v_{-}$ asymptotic to $u$ at $-\infty$. As remarked in \cite{BSS09}, it is enough to prove \eqref{spacetimeestimate} as \eqref{energyzero} follows as a consequence. They first established the existence of the wave operator, that is for any solution $v$ to the linear equation \eqref{linearequation}, there exists a unique solution $u$ to the non linear equation \eqref{waveequation} such that
$$\lim_{t\longrightarrow\infty}E_{0}(u-v;t)=0$$ 
Given \eqref{strichartz}, for any $\delta>0$ one may select $T$ large so that $\left\|v\right\|_{L^{5}\left(\left[T,\infty\right);L^{10}(\Omega)\right)}\leq\delta$. Given any $w(t,x)$ satisfying $\left\|w\right\|_{L^{5}\left(\left[T,\infty\right);L^{10}(\Omega)\right)}\leq\delta$, we have a unique solution to the linear problem 
$$\Box\widetilde{w}=-(v+w)^{5}$$
$$\lim_{t\longrightarrow\infty}E_{0}(\widetilde{w};t)=0$$
as the right hand side is in $L^{1}([T,\infty);L^{2}(\Omega))$. The estimate \eqref{strichartz} then also ensures that 
$$\left\|\widetilde{w}\right\|_{L^{5}\left(\left[T,\infty\right);L^{10}(\Omega)\right)}\leq c\left\|v+w\right\|^{5}_{L^{5}\left(\left[T,\infty\right);L^{10}(\Omega)\right)}\leq 32c\delta^{5}$$
Hence for $\delta$ sufficiently small, the map $w\longrightarrow\widetilde{w}$ is seen to be a contraction on the ball of radius $\delta$ in $L^{5}\left(\left[T,\infty\right);L^{10}(\Omega)\right)$. The unique fixed point $w$ can be uniquely extended over all of $\mathbb{R}\times\Omega$. Hence taking $u=v+w$ shows the existence of the wave operator.

To see that the wave operator is surjective, they used the $L^6$ decay estimate which we proved in Theorem \ref{L6decay} for our obstacle. 
This decay estimate establishes that the non linear effects of the solution map for the non linear equation \eqref{waveequation} 
diminish as time evolves.

By the result of Theorem \ref{L6decay}, given any $\epsilon>0$, there exists $T$ sufficiently large such that
$$\sup_{t\geq T}\left\|u(t,\cdot)\right\|_{L^{6}}<\epsilon$$
Hence for any $S>T$ we obtain the following for any solution $u$ to \eqref{waveequation}
\begin{align*}
\left\|u\right\|_{L^{5}([T,S];L^{10}(\Omega))}+\left\|u\right\|_{L^{4}([T,S];L^{12}(\Omega))}&\leq c\left(E+\left\|u^{5}\right\|_{L^{1}([T,S];L^{2}(\Omega))}\right)\\
&\leq cE+c\epsilon\left\|u\right\|^{4}_{L^{4}([T,S];L^{12}(\Omega))}
\end{align*} 
A continuity argument now yields $\left\|u\right\|_{L^{5}([T,\infty);L^{10}(\Omega))}+\left\|u\right\|_{L^{4}([T,\infty);L^{12}(\Omega))}\leq 2cE$ and by time reflection argument, \eqref{spacetimeestimate} follows. However, this implies that the linear problem 
$$\Box w=-u^{5}$$
$$\lim_{t\longrightarrow\infty}E_{0}(w;t)=0$$
admits a solution, showing that the wave operator is indeed surjective as $v=u-w$ is the desired solution to \eqref{linearequation}.



\begin{thebibliography}{50}
\bibitem{BG99}{Bahouri, H. and G\'{e}rard, P., }{\em High frequency approximation of solutions to critical nonlinear wave equations,} Amer. J. Math., Vol. 121, (1999), p. 131-175.

\bibitem{BS98}{Bahouri, H. and Shatah, J., }{\em Decay estimates for the critical semilinear wave equation,} Ann. Inst. Henri Poinar\'{e}, Vol. 15, nb. 6, (1998), p. 783-789.

\bibitem{BSS09}{Blair, M. D., Smith, H. F. and Sogge, C. D., }{\em Strichartz estimates for the wave equation on manifolds with boundary,} Annales de l'Institut Henri Poincare, 26 (2009), 18171829.

\bibitem{BK74}{Bloom, C. O., Kazarinoff, N. D., }{\em Local energy decay for a class of nonstar-shaped bodies,} Arch. Rat. Mech. Anal., Vol. 55, (1974), p. 73-85.

\bibitem{BK76}{Bloom, C. O., Kazarinoff, N. D., }{\em Short wave radiation problems in inhomogeneous media: Asymptotic solutions,} Lecture Notes in Mathematics 522. Springer-Verlag, Berlin et al., 1976.

\bibitem{BLP08}{Burq, N., Lebeau, G. and Planchon, F., }{\em Global existence for energy critical wave in 3-D domains,} J. Amer. Math. Soc., Vol. 21, (2008), p. 831-845.

\bibitem{BP09}{Burq, N. and Planchon, F., }{\em Global existence for energy critical waves in 3-D domains: Neumann boundary conditions,} Amer. J. Math., Vol. 131, Nb. 6, December 2009, p. 1715-1742.

\bibitem{G90} {Grillakis, M. G., }{\em Regularity and asymptotic behavior of the wave equation with a critical nonlinearity,} Ann. of Math., Vol. 132, (1990), p. 485-509.

\bibitem{G92} {Grillakis, M. G., }{\em Regularity for the wave equation with a critical nonlinearity,} Comm. Pure App. Math., Vol. 45, (1992), p. 749-774.

\bibitem{Ivrii} {Ivrii, V. Ja., }{\em Exponential decay of the solution of the wave equation outside an almost star-shaped region,} (Russian) Dokl. Akad. Nauk SSSR, Vol. 189, (1969), p. 938-940.

\bibitem{LMP63} {Lax, P. D., Morawetz, C. S. and Phillips, R. S., }{\em Exponential decay of solutions of the wave equation in the exterior of a star-shaped obstacle,} Comm. Pure Appl. Math., Vol. 16, (1963), p.477-486.

\bibitem{L87}{Liu, De-Fu, }{\em Local energy decay for hyperbolic systems in exterior domains, } J. Math. Anal. Appl., Vol. 128, (1987), p. 312-331.

\bibitem{M61}{Morawetz, C. S., } {\em The decay of solutions of the exterior initial-boundary value problem for the wave equation.} Comm. Pure Appl. Math., Vol. 14, (1961), p. 561-568.

\bibitem{M66}{Morawetz, C. S., }{\em Exponential decay of solutions of the wave equation, } Comm. Pure Appl. Math., Vol. 19, (1966), p.439-444.

\bibitem{M62}{Morawetz, C. S., }{\em The limiting amplitude principle, } Comm. Pure Appl. Math., Vol. 15, (1962), 349-362.

\bibitem{MRS77}{Morawetz, C. S., Ralston, J. V., Strauss, W. A., }{\em Decay of solutions of the wave equation outside nontrapping obstacles, } Comm. Pure Appl. Math., Vol. 30, (1977), 447-508.

\bibitem{SS93}{Shatah, J. and Struwe, M., }{\em Regularity results for nonlinear wave equations, } Ann. of Math., Vol. 138, (1993), p. 503-518.

\bibitem{SS94}{Shatah, J. and Struwe, M., }{\em Well-posedness in the energy space for semilinear wave equations with critical growth, } Internat. Math. Res. Notices, Vol. 7, (1994), p. 303-309.

\bibitem{SS95}{Smith, H. F. and Sogge, C. D., }{\em On the critical semilinear wave equation outside convex obstacles, } J. Amer. Math. Soc., Vol. 8, (1995), 879-916.

\bibitem{S75}{Strauss, W. A., }{\em Dispersal of waves vanishing on the boundary of an exterior domain,} Comm. Pure Appl. Math., Vol. 28, (1975), p. 265-278.

\end{thebibliography}
\end{document}